\documentclass{amsart}

\usepackage{amsmath, amsthm, amsfonts, mathrsfs, amsfonts}
\usepackage{bm}
\usepackage{amscd}
\usepackage{amssymb}
\usepackage{enumerate} %numeri elenchi
\usepackage{xy}

\newcommand{\F}{\mathbb{F}}

\newcommand{\Z}{\ensuremath{\mathbb{Z}}}

\newcommand{\cor}[1]{\mathcal{#1}}

\newcommand{\graffe}[1]{\{#1\}}

\DeclareMathOperator{\Syl}{Syl} %sylow
\DeclareMathOperator{\Aut}{Aut} %automorphism group
\DeclareMathOperator{\hc}{H} %cohomology
\DeclareMathOperator{\Res}{Res} %restriction map
\DeclareMathOperator{\ord}{ord} %order
\DeclareMathOperator{\nor}{N} %normalizer
\DeclareMathOperator{\Heis}{Heis} %heisenberg group
\DeclareMathOperator{\Br}{Br}
\DeclareMathOperator{\Sym}{Sym}
\DeclareMathOperator{\Int}{Int}

\newenvironment{sistema}
{\left\lbrace\begin{array}{@{}l@{}}}
{\end{array}\right.}
\newtheorem{definition}{Definition}
\newtheorem{lemma}[definition]{Lemma}
\newtheorem{theorem}[definition]{Theorem}
\newtheorem{proposition}[definition]{Proposition}

\newtheorem{theoremA}{Theorem}

\newtheorem{propositionA}[theoremA]{Proposition}

\newcounter{ex}[section]
\newenvironment{ex}[1][]{\refstepcounter{ex}\par\medskip\noindent \textbf{Example.} \rmfamily}{\medskip}

\begin{document}

%-------------------------------------------------------------------------
% editorial commands: to be inserted by the editorial office
%
%\firstpage{1} \volume{228} \Copyrightyear{2004} \DOI{003-0001}
%\seriesextraline{This is the Concrete Title of this Book\br H.E. R and S.T.C. W, Eds.}
%\firstpage{1}

\title[Evolving groups]
 {Evolving groups}

%----------Author 1
\author{Mima Stanojkovski}

\address{%
Fakult\"{a}t f\"{u}r Mathematik \\
Universit\"{a}t Bielefeld \\
Postfach 100131 \\
D-33501 Bielefeld \\
Germany }

\email{mstanojk@math.uni-bielefeld.de}

\thanks{This paper collects the results from the author's master thesis \cite{evolving}, which she wrote at Leiden University under the supervision of Prof. Hendrik Lenstra. }



%\date{today}

\begin{abstract}
The class of evolving groups is defined and investigated, as well as their connections to examples in the field of Galois cohomology. Evolving groups are proved to be Sylow Tower groups in a rather strong sense. In addition, evolving groups are characterized as semidirect products of two nilpotent groups of coprime orders where the action of one on the other is via automorphisms that map each subgroup to a conjugate.  
\end{abstract}

%%% ----------------------------------------------------------------------
\maketitle
%%% ----------------------------------------------------------------------

\section{Introduction}

\noindent
Let $G$ be a group. We say that $G$ is \emph{evolving} if it is finite and if, for every prime number $p$ and for every $p$-subgroup $I$ of $G$, there exists a subgroup $J$ of 
$G$ that contains $I$ and such that $|G:J|$ is a $p$-power and $|J:I|$ is coprime to $p$.
In Section \ref{section 1} we show that normal subgroups and quotients of evolving groups are evolving, while arbitrary subgroups are not necessarily evolving. In Section \ref{section T1} we prove the following result involving Tate cohomology groups. % (see e.g. \cite[Ch.~$\mathrm{IV}.6$]{casfrol}).

\begin{theoremA}\label{T1}\label{starter}
Let $G$ be a finite group. Then the following are equivalent.
\begin{itemize}
 \item[$i$.] For every $G$-module $M$, for every integer $q$, and for every $c\in\widehat{\hc}{^{q}}(G,M)$, the minimum of the set
           $\{|G:H| : H\leq G \ \ \text{with} \ \ c\in\ker\Res_H^G\}$ coincides with its greatest common divisor.
 \item[$ii$.] For every $G$-module $M$ and for every element $c\in\widehat{\hc}{^{0}}(G,M)$, the minimum of the set
           $\{|G:H| : H\leq G \ \ \text{with} \ \ c\in\ker\Res_H^G\}$ coincides with its greatest common divisor.
 \item[$iii$.] The group $G$ is evolving.         
\end{itemize}
\end{theoremA}

\noindent
Condition ($i$) from Theorem \ref{T1} is inspired, via Galois cohomology, by the following examples (for more information concerning the connection to cohomology groups see e.g. \cite[Ch.~$\mathrm{V}.14.2$, $\mathrm{VIII}$, $\mathrm{IX}$]{berhuy}).

\begin{ex}
Let $k$ be a field and let $n$ be a positive integer.
Moreover, let $a$ be a non-zero element of $k$. Then the least degree of the irreducible factors of $x^n-a$ over $k$ divides all other degrees. 
\end{ex} 

\begin{ex} Let $k$ be a field and let $\Br(k)$ denote the Brauer group of $k$, i.e. the group of similarity classes of central simple algebras over $k$, endowed with the multiplication $\otimes_k$.
If $[A]\in\Br(k)$, then an extension $\ell/k$ is said to \emph{split} $A$ if $[A\otimes_k\ell]=[\ell]$ and the minimal degree of finite separable extensions of $k$ that split $A$ over $k$ divides all other degrees (see for example \cite[Ch.~$4.5$]{algebras}). 
\end{ex}

\begin{ex} Let $k$ be a field and let $C$ be a smooth projective absolutely irreducible curve of genus $1$ over $k$. As a consequence of the Riemann-Roch theorem, as explained for example in \cite[\S $2$]{langtate}, the least degree of the finite extensions of $k$ for which $C$ has a rational point divides all other degrees. 
\end{ex}

\noindent
In the course of our investigation of evolving groups we come upon a peculiar family of Sylow Tower groups: we say that a finite group $G$ is \emph{prime-intense} if it possesses a collection $(S_p)_{p| |G|}\in\prod_{p| |G|}\Syl_p(G)$ of Sylow subgroups with the property that, for all primes $p>q$ dividing $|G|$, one has $S_q\subseteq\nor_G(S_p)$ and, for each subgroup $H$ of $S_p$ and each $x\in S_q$, the subgroups $xHx^{-1}$ and $H$ are $S_p$-conjugate. In Section \ref{section primeintense} we prove the next two results.

\begin{theoremA}\label{behaviour}
A finite group is evolving if and only if it is prime-intense.
\end{theoremA}

\begin{propositionA}\label{supersolvable}
Evolving groups are supersolvable.
\end{propositionA}

\noindent
Given a group $G$, we say that an automorphism of $G$ is \emph{intense} if it maps each subgroup of $G$ to a conjugate and we denote by $\Int(G)$ the collection of such automorphisms. We prove some basic facts about intense automorphisms in Section $4$ and,
in Section \ref{sec-structure}, we prove the following.

\begin{theoremA}\label{evolvingintense}
Let $G$ be a finite group. Then the following are equivalent.
\begin{itemize}
 \item[$i.$] The group $G$ is evolving.
 \item[$ii.$] There exist nilpotent groups $N$ and $T$ of coprime orders and a group homomorphism $\phi:T\rightarrow\Int(N)$ such that $G=N\rtimes_{\phi}T$.
\end{itemize} 
\end{theoremA}

\noindent
Theorems \ref{T1} and \ref{evolvingintense} can be generalized to profinite groups. Indeed, pro-evolving groups can be characterized by an analogue of Theorem \ref{T1}($i$) where Tate groups are replaced by cohomology groups of positive degree, modules by discrete modules, and subgroups by open ones. Extending Theorem \ref{evolvingintense}, one can prove that pro-evolving groups are semi-direct products of two pro-nilpotent groups of coprime orders where the action is given by automorphisms sending each closed subgroup to a conjugate. Given the origins of evolving groups, it would be interesting to determine which pro-evolving groups can be realized as absolute Galois groups of a field.

\section{Evolving groups}\label{section 1}

\noindent
The current section is dedicated to subgroups and quotients of evolving groups. We prove that quotients of evolving groups are evolving and that the same applies to their normal subgroups, while arbitrary subgroups of evolving groups are not necessarily evolving.
We present an example of a an evolving group with a non-evolving subgroup after proving Proposition \ref{property preserved}.

\begin{definition}\label{p-evolution}
Let $G$ be a finite group and let $p$ be a prime number. Let $I$ be a $p$-subgroup of $G$. We say that $J$ is a \emph{$p$-evolution of $I$ in $G$} if $J$ is a subgroup of $G$ with the following properties:
\begin{itemize}
 \item[$i.$] the subgroup $I$ is contained in $J$; and
 \item[$ii.$] the index $|G:J|$ is a power of $p$; and
 \item[$iii.$] the prime $p$ does not divide $|J:I|$.
\end{itemize}
\end{definition}

\noindent
In a finite group $G$, each Sylow $p$-subgroup has a $p$-evolution, namely $G$. 
In general, however, $p$-evolutions of arbitrary $p$-subgroups might not exist. 
For instance, the trivial subgroup of the alternating group $A_5$ has a $5$-evolution, but no $2$-evolution. 
The finite groups in which every $p$-subgroup has a $p$-evolution are exactly the evolving groups.

\begin{lemma}\label{property preserved}
Let $G$ be an evolving group and let $N$ be a normal subgroup of $G$. 
Then both $N$ and $G/N$ are evolving groups. 
\end{lemma}

\begin{proof}
Let $p$ be a prime number. Then, for every $p$-subgroup $I$ of $N$ and for every $p$-evolution $K$ of $I$ in $G$, the subgroup $J=N\cap K$ is a $p$-evolution of $I$ in $N$. This proves that $N$ is evolving. 
Let now $I/N$ be a subgroup of $G/N$ such that $|I:N|$ is a $p$-power. 
Let $I_p$ be a Sylow $p$-subgroup of $I$ and let $J$ be a $p$-evolution of $I_p$ in $G$; then $I_p$ belongs to 
$\Syl_p(J)$. Let  moreover $q\neq p$ be a prime number and let $J_q\in\Syl_q(J)$. 
Then $J_q$ belongs to $\Syl_q(G)$ and thus $J_q\cap N$ is a Sylow subgroup both of $N$ and, $|I:N|$ being a $p$-power, of $I$. 
As a consequence, $J$ contains $I$ and $J/N$ is a $p$-evolution of $I/N$ in $G/N$. It follows that $G/N$ is evolving.
\end{proof}

\noindent
Let $p$ be an odd prime number and define the group
\[
G=
\left\{
\begin{pmatrix} 
\lambda & u & v \\ 
0 & 1 & w \\
0 & 0 & \lambda^{-1}
\end{pmatrix} 
\ :\  u,v,w\in\F_p, \lambda\in\F_p^*
\right\},%\cong\Heis(\F_p)\rtimes\F_p^*
\]
which is isomorphic to a semidirect product $\Heis(\F_p)\rtimes\F_p^*$. 
With $N=\Heis(\F_p)$ and $T$ the set of diagonal matrices of $G$, Theorem \ref{evolvingintense} yields that $G$ is evolving.
On the other hand, the subgroup $W$ 
%\[
%W=
%\left\{
%\begin{pmatrix} 
%\lambda & u & v \\ 
%0 & 1 & 0 \\
%0 & 0 & \lambda^{-1}
%\end{pmatrix} 
%\ :\  u,v\in\F_p, \lambda\in\F_p^*
%\right\}%\cong(W\cap\Heis(\F_p))\rtimes\F_p^*
%\]
of $G$ corresponding to the collection of matrices with $w$-entry equal to $0$ is not evolving. To show so, one can display a subgroup with no evolution or build a contradiction to Theorem \ref{evolvingintense} using results from Section \ref{section intense}.

\section{The proof of Theorem \ref{T1}}\label{section T1}

\noindent
The aim of the present section is to give the proof of Theorem \ref{starter}. We refer 
to Chapter $\mathrm{IV}$ from \cite{casfrol} for basic concepts concerning group cohomology and Tate cohomology groups in particular.
In order to prove Theorem \ref{starter}, we reduce to the case in which the cocycles have order equal to a prime power. We will use that, if $G$ is a finite group, $M$ a $G$-module, and $(S_p)_{p | |G|}$ a collection of Sylow $p$-subgroups of $G$, then, for example as a consequence of Corollary $3$ from \cite[Ch.~$\mathrm{IV}$]{casfrol}, the homomorphism $\widehat{\hc}{^{q}}(G,M)\rightarrow \bigoplus_{p | |G|}\widehat{\hc}{^{q}}(S_p,M)$ that is defined by $c\mapsto (\Res^G_{S_p}(c))_{p| |G|}$ is injective.

\begin{lemma}\label{prime order starter}
Let $G$ be a finite group, let $M$ be a $G$-module, and let $q$ be an integer. 
Let moreover $p$ be a prime number, let $c\in\widehat{\hc}{^{q}}(G,M)$ be of order a power of $p$, and assume that
every $p$-subgroup of $G$ has a $p$-evolution in $G$.  
Then the minimum of the set ${\{|G:H| : H\leq G \ \ \text{with} \ \ c\in\ker\Res_H^G\}}$ coincides with its greatest common divisor and it is a $p$-th power. 
\end{lemma}

\begin{proof}
Let $r\neq p$ be a prime number and let $S_r\in\Syl_r(G)$. Being annihilated by both a power of $p$ and a power of $r$, the restriction of $c$ to $S_r$ is zero.    
Let now $I_p$ be a $p$-subgroup of maximal order of $G$ with the property that $\Res_{I_p}^G(c)=0$ and let $J_p$ be a $p$-evolution of $I_p$ in $G$. 
The subgroup $I_p$ being a Sylow $p\text{-subgroup}$ of $J_p$, it follows that $c$ is zero when restricted to a selected family of Sylows of $J_p$ and thus ${\Res_{J_p}^G(c)=0}$. Moreover, from the maximality of $I_p$, it follows that $|G:J_p|$ is the greatest common divisor of the set 
${\{|G:H| : H\leq G \ \ \text{with} \ \ c\in\ker\Res_H^G\}}$, which is then also a $p$-th power.  
\end{proof}

\begin{proposition}\label{starter:3to1}
Let $G$ be an evolving group. Let, moreover, $M$ be a $G$-module, let $q$ be an integer, and let $c\in\widehat{\hc}{^{q}}(G,M)$. 
Then the minimum and the greatest common divisor of the set ${\{|G:H| : H\leq G \ \ \text{with} \ \ c\in\ker\Res_H^G\}}$ coincide. 
\end{proposition}

\begin{proof}
Write $c=\sum_{p | |G|}c_p$, with $c_p$ in the $p$-primary component of $\widehat{\hc}{^{q}}(G,M)$. 
By Lemma \ref{prime order starter}, for every prime $p$, there exists a subgroup $J_p\leq G$ such that 
$|G:J_p|$ is a $p$-power and it is both the minimum and the greatest common divisor of the set ${\{|G:H| : H\leq G \ \ \text{with} \ \ c_p\in\ker\Res_H^G\}}$.
For each prime $p$, fix $J_p$ and define $L=\bigcap_{p | |G|}J_p$. Then $\Res_L^G(c)=0$ and $|G:L|$ divides the index in $G$ of every subgroup on which the restriction of $c$ vanishes.
Indeed, if $K$  is a subgroup such that $\Res^G_K(c)=0$, then, for every prime $p$, one has $\Res_{K}^G(c_p)=0$ from which it follows that $|G:J_p|$ divides $|G:K|$ and thus
$|G:L|=\prod_{p| |G|}|G:J_p|$ divides $|G:K|$. 
\end{proof}

\noindent
For the second part of the proof of Theorem \ref{starter}, we will use a convenient description of the Tate groups in degree $0$ in terms of permutation modules.

\begin{definition}
Let $G$ be a group and let $X$ be a finite $G$-set. 
The \emph{permutation module associated to $X$} is  the set 
$M=\Z^X$ together with the left action $G \rightarrow \Sym(M)$ that is defined by 
$g\mapsto (f\mapsto (x\mapsto f(g^{-1}x))).$
\end{definition}

\noindent
Let $G$ be a group, let $X$ be a finite $G$-set, and let $M$ be the permutation module associated to $X$. Then $f\in\hc^0(G,M)$ if and only if, for any choice of $x\in X$ and $g\in G$, 
one has $f(x)=(gf)(x)=f(g^{-1}x)$ and therefore $\hc^0(G,M)$ consists precisely of the functions $X\rightarrow\Z$ that are constant on each $G$-orbit.

\begin{proposition}\label{tate permutation}
Let $G$ be a finite group and let $X$ be a finite $G$-set. Let $M$ be the permutation module associated to $X$. Then the map
$$\gamma_G:\widehat{\hc}{^{0}}(G,M)\longrightarrow\bigoplus_{Gx\in G\backslash X}\Z/|G_x|\Z$$  that is defined by 
$[f]\mapsto(f(x)\bmod|G_x|)_{Gx\in G\backslash X}$ is an isomorphism of groups.
Moreover, if $H$ is a subgroup of $G$, then the following diagram is commutative
\vspace{5pt}
\begin{equation}
\begin{CD}
\widehat{\hc}{^{0}}(G,M)\ \ \ \ \ \ @>\gamma_G>>  \bigoplus_{Gx\in G\backslash X}\Z/|G_x|\Z \\
@V{\Res_H^G}VV                            @VV{\pi}V \\
\widehat{\hc}{^{0}}(H,M)\ \ \ \ \ \ @>\gamma_H>> \ \ \ \  \bigoplus_{Gx\in G\backslash X}\ \bigoplus_{Hy\in H\backslash Gx}\Z/|H_y|\Z
\end{CD}\nonumber
\end{equation}
\vspace{5pt}\\
\noindent
where $\pi$ is the projection map, which, restricted to each $Gx$-th direct summand, sends the element $m\bmod|G_x|$ to $(m\bmod|H_y|)_{Hy\in H\backslash Gx}$.
\end{proposition}

\begin{proof}
With standard arguments it can be seen that $\gamma_G$ is a well-defined surjective group homomorphism and that the above diagram is commutative.
We prove injectivity.
To this end, we recall that $\widehat{\hc}{^{0}}(G,M)$ is the cokernel of the homomorphism $M\rightarrow\hc^0(G,M)$ that is defined by $m\mapsto\sum_{g\in G}gm$.
Let now $f\in\hc^0(G,M)$ be such that $\gamma_G([f])=0$. Then, for all $Gx\in G\backslash X$, the order of $G_x$ divides $f(x)$, 
i.e. there exists an element $\phi_x\in\Z$ such that $f(x)=|G_x|\phi_x$. Choose now a representative $x$ for each orbit $Gx$ and define
$\phi:X\rightarrow\Z$ by sending $x$ to
$\phi_x$ and all other orbit elements to $0$. Then $[f]=[(\sum_{g\in G}g)\phi]=[0]$ and so $\gamma_G$ is injective.
\end{proof}

\noindent
Until the end of Section \ref{section 1}, we will identify $c\in\widehat{\hc}{^{0}(G,M)}$, for a fixed permutation module $M$, with its image under
$\gamma_G$ and, similarly, we will identify the restriction of $c$ to a subgroup $H$ of $G$ with $\pi\gamma_{G}(c)$.

\begin{lemma}\label{case}
Let $G$ be a finite group and let $H,K$ be subgroups of $G$. Set ${M=\Z^{G/K}}$ and let $m\in\widehat{\hc}{^{0}}(G,M)$. Then the following are equivalent.
\begin{itemize}
 \item[$i.$] One has $\Res^G_H(m)=0$.
 \item[$ii.$] For all $g\in G$, there exists $m_g\in \Z/|K|\Z$ such that $|gKg^{-1}\cap H|m_g=m$. 
\end{itemize}
\end{lemma}

\begin{proof}
This follows directly from Proposition \ref{tate permutation} and the fact that, for each $g\in G$, one has $H_{gK}=H\cap G_{gK}=H\cap gKg^{-1}$.
\end{proof}

\begin{lemma}\label{cocycle}
Let $p$ be a prime number, let $\ord_p$ denote the $p$-adic valuation, and let $\alpha\in\Z_{\geq0}$.
Let moreover $G$ be a finite group, let $S$ be a Sylow $p$-subgroup of $G$, and let $I$ be a subgroup of $G$ of order $p^\alpha$. Define 
$$\cor{L}=\{L\leq G : |L|=|I|\ \ \text{and}\ \ L\neq gIg^{-1}\ \ \text{for all}\ \ g\in G\}$$
and 
$$M=\Z^{G/S}\bm{\oplus}\big(\bigoplus_{L\in\cor{L}}\Z^{G/L}\big).$$
Let $H$ be a subgroup of $G$ and define 
$c=(m,(m_L)_{L\in\cor{L}})\in\widehat{\hc}{^{0}}(G,M)$ by
\[
\begin{sistema}
 m=p^\alpha\bmod |S| \\ 
 m_L=p^{\alpha-1}\bmod |L|,\ \ \text{for all}\ \ L\in\cor{L}.
\end{sistema}
\]
Then $\Res_H^G(c)=0$ if and only if $\ord_p|H|\leq\alpha$ and, if equality holds, every Sylow $p$-subgroup of $H$ is conjugate to $I$ in $G$.
\end{lemma}

\begin{proof}
Thanks to Lemma \ref{case}, one has $\Res_H^G(m)=0$ if and only if, for all $g\in G$, the order of 
$H\cap gSg^{-1}$ divides $p^\alpha$ and thus a necessary condition for $\Res^G_H(c)$ to be zero is that $\ord_p|H|\leq\alpha$: we now assume so and prove that such condition is also sufficient.
If $\alpha=0$ we are done, otherwise fix $L\in\cor{L}$. 
Then $\Res_H^G(m_L)=0$ if and only if, for all $g\in G$, the order of $H\cap gLg^{-1}$ divides $p^{\alpha-1}$. If $\ord_p|H|\leq\alpha-1$, this last condition is 
always satisfied. On the other hand, if $\ord_p|H|=\alpha$, then $H\cap gLg^{-1}$ divides $p^{\alpha-1}$ if and only if $H$ does not contain conjugates of $L$.
In such case, 
it follows from the definition of $\cor{L}$ that
the Sylow $p$-subgroups of $H$ are all conjugate to $I$. 
\end{proof}

\begin{proposition}\label{tate0}
Let $G$ be a finite group such that, for every $G$-module $M$ and for every $c\in\widehat{\hc}{^{0}}(G,M)$, the minimum and the greatest common divisor of
the set ${\{|G:H| : H\leq G \ \ \text{with} \ \ c\in\ker\Res_H^G\}}$ coincide. Then 
$G$ is evolving. 
\end{proposition}

\begin{proof}
Let $p$ be a prime and let $I$ be a $p$-subgroup of $G$. We will construct a {$p$-evolution} $J$ of $I$.
Define $\alpha=\ord_p|I|$. Let $S\in\Syl_p(G)$ and let
$\cor{L}$ and $M$ be as in Lemma \ref{cocycle}. 
Thanks to Proposition \ref{tate permutation}, we identify
$\widehat{\hc}{^{0}}(G,M)$ with $\Z/|S|\Z\bm{\oplus}\big(\bigoplus_{L\in\cor{L}}\Z/|I|\Z\big)$.
Let $c\in\widehat{\hc}{^{0}}(G,M)$ be as in Lemma \ref{cocycle}, which then yields
$\Res_I^G(c)=0$, but the restriction of $c$ to subgroups of higher $p$-power order is non-zero. Moreover, $\Res_{S_r}^G(c)=0$ for every prime number $r\neq p$ and for every $S_r\in\Syl_r(G)$.
It follows that the greatest common divisor of the set ${\{|G:H| : H\leq G \ \ \text{with} \ \ c\in\ker\Res_H^G\}}$ is equal to $p^{\ord_p|G|-\alpha}$ and so there exists
a subgroup $K$ with index ${|G:K|}=p^{\ord_p|G|-\alpha}$ such that $\Res_{K}^G(c)=0$. Fix such $K$. Then, thanks to Lemma \ref{cocycle}, the subgroup $K$ does not contain any $p$-subgroup of order $p^\alpha$ 
which is not $G\text{-conjugate}$ to $I$.
As a result, there exists $g\in G$ such that $gIg^{-1}\leq K$, making $J= g^{-1}Kg$ a $p$-evolution of $I$.
\end{proof}

\noindent
Theorem \ref{starter} follows from Propositions \ref{starter:3to1} and \ref{tate0}.

\section{Intense automorphisms}\label{section intense}

\noindent
Let $G$ be a group. We call an automorphism $\alpha$ of $G$ \emph{intense} if, for each subgroup $H$ of $G$, there exists $x\in G$ such that $\alpha(H)=xHx^{-1}$ and we denote the collection of intense automorphisms of $G$ by $\Int(G)$. The next result collects some basic properties of intense automorphisms, while Lemma \ref{intense vector space} provides us with examples of intense automorphisms that are not inner.

\begin{lemma}\label{intense properties}
Let $G$ be a group and let $N$ be a normal subgroup of $G$. Then the following hold.
\begin{itemize}
 \item[$i$.] There is a natural map $\Int(G)\rightarrow\Aut(N)$, given by $\alpha\mapsto \alpha_{|N}$. %If $N$ is central, then the image of such map is contained in $\Int(N)$.
 \item[$ii$.] The natural projection $G\rightarrow G/N$ induces a homomorphism of groups $\Int(G)\rightarrow\Int(G/N)$.               
\end{itemize}
\end{lemma}

\begin{lemma}\label{intense vector space}
Let $p$ be a prime number and let $V$ be a vector space over the finite field $\F_p$ of $p$ elements. 
Then one of the following holds. 
\begin{itemize}
 \item[$i$.] One has $V=0$.
 \item[$ii$.] The natural map $\F_p^*\rightarrow\Aut(V)$ induces an isomorphism $\F_p^*\rightarrow\Int(V)$.
\end{itemize}
\end{lemma}

\begin{proof}
Assume $V\neq0$. Since $V$ is abelian, every one-dimensional subspace of $V$ is stable under the action of $\Int(V)$. It follows that, for all 
$v\in V\setminus\graffe{0}$ and $\alpha\in\Int(V)$, there exists (a unique) $\mu(\alpha,v)\in\F_p^*$ such that $\alpha(v)=\mu(\alpha,v) v$. Moreover, such $\mu(\alpha,v)$ is independent of the choice of $v$ thanks to the linearity of $\alpha$.
We fix $v\in V\setminus\graffe{0}$ and define $\mu:\Int(V)\rightarrow\F_p^*$ by $\alpha\mapsto\mu(\alpha,v)$, which is an injective homomorphism by construction. Moreover, $\mu$ is surjective, because scalar multiplication by any element of $\F_p^*$ is intense.
\end{proof}

\begin{lemma}\label{equivalent intense coprime}
Let $G$ be a finite group and let $A$ be a subgroup of $\Aut(G)$ of order coprime to that of $G$. Then the following are equivalent.
\begin{itemize}
 \item[$i.$] The group $A$ is contained in $\Int(G)$.
 \item[$ii.$] For each subgroup $H$ of $G$, there exists $x\in G$ such that, for each $\alpha\in A$, one has $\alpha(xHx^{-1})=xHx^{-1}$.
\end{itemize}
\end{lemma}

\begin{proof}
Showing $(ii)\Rightarrow(i)$ is an application of the definitions; we prove that $(i)\Rightarrow (ii)$. Write $X=\{xHx^{-1} : x\in G\}$. Then $G$ acts on $X$ by conjugation and $A$ acts on $X$ via intense automorphisms. The actions are compatible and the action of $G$ is transitive. In this situation Glauberman's Lemma \cite[Lemma $3.24$]{isaacs} yields that there exists an element of $X$ that is fixed by $A$. 
\end{proof}

\noindent
We remark that, with the notation from Lemma \ref{equivalent intense coprime}, Glauberman's Lemma would require that at least one between $G$ and $A$ is solvable. To avoid the problem, we could rely on the Odd Order Theorem, however in our applications $G$ will always be a $p$-group and therefore solvable.

\section{The proof of Theorem \ref{behaviour}}\label{section primeintense}

\noindent
In this section we show that evolving groups are supersolvable and give the proof of Theorem \ref{behaviour}.

\begin{definition}
Let $G$ be a finite group and let $\pi$ be the set of prime divisors of $|G|$.
A \emph{Sylow family of $G$} is 
an element $(S_p)_{p\in\pi}\in\prod_{p\in\pi}\Syl_p(G)$ with the property that, if $p,q\in\pi$ with $q<p$, then
$S_q$ normalizes $S_p$ in $G$.
\end{definition}

\noindent
To lighten the notation, given a Sylow family $\cor{S}=(S_q)_{q\in\pi}$ of $G$ and a prime $p$ dividing the order of $G$,
we will write 
$$T_p^{\cor{S}}=\langle S_q\ |\ q<p\rangle \ \ \ \text{and}\ \ \ L_p^{\cor{S}}=\langle S_q\ |\ p<q\rangle.$$
It is an easy exercise to show that, if $\bar{\pi}$ is a subset of $\pi$, then the subgroup 
$\langle S_q\ |\ q\in\bar{\pi}\rangle$ has order $\prod_{q\in\bar{\pi}}|S_q|$. 
As a result, for each prime number $p$ in $\pi$, one has that $G=L_p^{\cor{S}}\rtimes(S_p\rtimes T_p^{\cor{S}})$. We will simply write $G=L_p^{\cor{S}}\rtimes S_p\rtimes T_p^{\cor{S}}$.

\begin{lemma}
Let $G$ be a finite group and let ${\cor{S}}=(S_p)_{p\in\pi}$ and $\cor{R}=(R_p)_{p\in\pi}$ be Sylow families of $G$. Then there exists $x\in G$ such that, for all $p\in\pi$, one has $R_p=xS_px^{-1}$. 
\end{lemma}

\begin{proof}
We work by induction on the order of $G$.
To this end, let $p$ be the largest prime number dividing the order of $G$. Then $S_p=R_p$ and so the Schur-Zassenhaus theorem yields the existence of $g\in G$ such that $T_p^{\cor{R}}=gT_p^{\cor{S}}g^{-1}$. Fix such $g$.
Then $(R_q)_{q<p}$ and $(gS_qg^{-1})_{q<p}$ are Sylow families of $T_p^{\cor{R}}$ and so, by the inductive hypothesis, we are done.
\end{proof}

\noindent
We call a finite group $G$ \emph{prime-intense} if it has a Sylow family $\cor{S}=(S_p)_{p\in\pi}$ such that, for every prime $p\in\pi$, the image of the action $T_p^\cor{S}\rightarrow\Aut(S_p)$ is contained in $\Int(S_p)$.

\begin{proposition}\label{behaviour:2to1}
Every finite prime-intense group is evolving.
\end{proposition}

\begin{proof}
Let $G$ be a finite prime-intense group and let $\cor{S}=(S_q)_{q\in\pi}$ be a Sylow family of $G$.
Let moreover $p$ be a prime number and let $I$ be a $p$-subgroup of $G$, which we assume, without loss of generality, to be contained in $S_p$. The group $G$ being prime-intense, Lemma \ref{equivalent intense coprime} yields that there exists $s\in S_p$ such that $T_p^\cor{S}$ normalizes 
$sIs^{-1}$. Fix $s$ and define $K=L_p^\cor{S}\rtimes sIs^{-1}\rtimes T_p^\cor{S}$. Then $|G:K|$ is a $p$-power and
$|K|/|I|$ is not divisible by $p$. It follows that the subgroup $J=s^{-1}Ks$ is a $p$-evolution of $I$ in $G$ and so $G$ is evolving.
\end{proof}

\begin{lemma}\label{sylows}
Every evolving group has a Sylow family.
\end{lemma} 

\begin{proof}
Let $G$ be an evolving group. We will work by induction on the order of $G$. 
Assume without loss of generality that $G$ is non-trivial and let $p$ and $r$ be respectively the largest and smallest prime dividing $|G|$. We claim that $G$ has a unique Sylow $p$-subgroup.
If $p=r$, we are clearly done so we assume $p>r$.
Let $R$ be a Sylow $r\text{-subgroup}$ of $G$ and let $T$ be a subgroup of index $r$ in $R$. The group $G$ is evolving and so  
there is an $r$-evolution $J$ of $T$ in $G$, which has then index $r$ in $G$. 
The prime $r$ being the smallest prime dividing the order of $G$, the subgroup $J$ is normal and hence, by Lemma \ref{property preserved}, evolving. 
By the induction hypothesis, there is a unique Sylow $p$-subgroup $P$ in $J$, which is also the unique $p$-Sylow of $G$.
This proves the claim.
Let now $S_p$ be the unique Sylow $p$-subgroup of $G$ and let $T_p$ be a complement of it in $G$, as given by the Schur-Zassenhaus theorem. Lemma \ref{property preserved} yields that $T_p$ is evolving and so, thanks to the induction hypothesis, $T_p$ has a Sylow family $(S_q)_{q\in\pi\setminus\{p\}}$. As a consequence, the family $(S_q)_{q\in\pi}$ is a Sylow family of $G$. 
\end{proof}

\begin{proposition}\label{ev to corr}
Every evolving group is prime-intense.
\end{proposition}

\begin{proof}
Let $G$ be an evolving group and let $\cor{S}=(S_q)_{q\in\pi}$ be a Sylow family of $G$, as given by Lemma \ref{sylows}. Assume without loss of generality that $G$ is non-trivial, let $p$ be the largest prime dividing 
$|G|$ and write $G=S_p\rtimes T_p^\cor{S}$.
Let moreover $I$ be a subgroup of $S_p$ and let $J$ be a $p$-evolution of $I$ in $G$. We will exibit an element $s\in S_p$ such that $T_p^\cor{S}$ is contained in $\nor_G(sIs^{-1})$. 
By comparing orders, we see that $G=S_pJ$ and that $S_p\cap J=I$. As a consequence, $I$ is normal in $J$ and so, by the Schur-Zassenhaus theorem, $I$ has a complement $T$ in $J$, which is also a complement of $S_p$ in $G$. 
Moreover, there is an element $s\in S_p$ such that $T_p^\cor{S}=sTs^{-1}$ and hence $sIs^{-1}$ is normalized by $T_p^\cor{S}$. 
Let now $q$ be any prime number dividing the order of $G$ and write $G=L_q^\cor{S}\rtimes S_q\rtimes T_q^\cor{S}$. The subgroup $L_q^\cor{S}$ being normal in $G$, Lemma \ref{property preserved} ensures that $S_q\rtimes T_q^\cor{S}$ is evolving. 
Now $q$ is the largest prime which divides the order of $S_q\rtimes T_q^\cor{S}$ and so, for every subgroup $I$ of $S_q$,
there exists $s\in S_q$ such that $T_q^\cor{S}$ normalizes $sIs^{-1}$. We are done thanks to Lemma \ref{equivalent intense coprime}.
\end{proof}

\noindent
Theorem \ref{behaviour} follows from Propositions \ref{behaviour:2to1} and \ref{ev to corr}.
\vspace{8pt}\\
\noindent
We conclude the present section by giving the proof of Proposition \ref{supersolvable}. To this end, let $G$ be an evolving group, which is prime-intense by Proposition \ref{ev to corr}, and let $\cor{S}=(S_q)_{q\in\pi}$ be a Sylow family of $G$.
To prove that $G$ is supersolvable, we will work by induction on its order. 
If $G$ is the trivial group, then $G$ is supersolvable, so we assume that $G$ is non-trivial and we let $p$ be the largest prime 
dividing $|G|$. Then $S_p$ is normal in $G$ and $T_p^\cor{S}$ is prime-intense. It follows from the induction hypothesis that $T_p^\cor{S}$ is supersolvable.
Let now the normal series $1=I_0\leq I_1\leq\dots\leq I_r=S_p$ be such that, for every $j\in\{0,\dots,r-1\}$, one has $|I_{j+1}:I_j|=p$ and fix $j\in\{0,\dots,r\}$. Thanks to Lemma \ref{equivalent intense coprime}, there exists $s\in S_p$
such that $T_p^{\cor{S}}$ normalizes $sI_js^{-1}=I_j$. Since
$G=S_p\rtimes T_p^{\cor{S}}$, the subgroup $I_j$ is normal in $G$ and thus $G$ is supersolvable.

\section{The proof of Theorem \ref{evolvingintense}}\label{sec-structure}

\noindent
In the present section we prove Theorem \ref{evolvingintense} by building a graph incorporating some structural properties of evolving groups, which we have proven to be prime-intense in Section \ref{section primeintense}.
\vspace{8pt} \\ 
\noindent
Let $G$ be an evolving group and let $\pi$ be the collection of prime numbers dividing the order of $G$. Let moreover $\cor{S}=(S_p)_{p\in\pi}$ be a Sylow family of $G$. Then Lemmas \ref{intense properties} and \ref{intense vector space} yield, for each $p\in\pi$, a sequence of maps
\[
T_p^{\cor{S}}\rightarrow \Int(S_p) \rightarrow\Int(S_p/\Phi(S_p))\rightarrow \F_p^*,
\]
whose composition $T_p^{\cor{S}}\rightarrow\F_p^*$ we denote by $\lambda_p$.
We define the graph $\Gamma=(V,E)$ associated to $G$ to be the directed graph with
$$V=\pi \ \ \ \text{and}\ \ \ E=\graffe{(q,p) : p,q\in V,\ q<p,\ \lambda_p(S_q)\neq 1}.$$
The source and target maps $s,t:E\rightarrow V$ are respectively defined, for each $(q,p)\in E$, by 
$s(q,p)=q$ and by $t(q,p)=p$. We define additionally 
\[\pi_s=s(E),\ \pi_t=t(E) \ \ \text{and} \ \ \pi_0=\pi\setminus(\pi_s\cup\pi_t).
\] 
Note that the the graph associated to $G$ does not depend on the choice of $\cor{S}$. In particular, the sets $\pi_s$, $\pi_t$, and $\pi_0$ are uniquely determined by $G$ and they form a partition of $\pi$, as a consequence of the following result.

\begin{proposition}\label{no consecutive edges}
Let $G$ be an evolving group and let $\Gamma=(V,E)$ be the graph associated to it. 
Then $\pi_s\cap\pi_t=\emptyset$, i.e. there are no consecutive edges in $\Gamma$. 
\end{proposition}

\begin{proof}
Let $p,q,r\in V$ be such that $p>q>r$ and let $(q,p)$ be an element of $E$. We will show that $\lambda_q(S_r)$ is trivial. Since $(q,p)\in E$, the subgroup 
$\lambda_p(S_q\rtimes S_r)$ is non-trivial: we choose $g\in S_r$ and we choose $s\in S_q$ such that $s$ does not map to the identity in $\F_p^*$. Since $\F_p^*$ is abelian, the element $[s,g]$ maps to $1$ and so
$gsg^{-1}\equiv s\bmod(\ker\lambda_p\cap S_q)$.
The group $\lambda_p(S_q)$ being non-trivial, $\ker\lambda_p\cap S_q$ is a proper subgroup of $S_q$ and therefore so is 
the subgroup $L=(\ker\lambda_r\cap S_q)\Phi(S_q)$. Since $gsg^{-1}\equiv s\bmod L$, the group $S_q/L$ is a non-zero quotient space of $S_q/\Phi(S_q)$ on which 
the action of $S_r$ is trivial. Lemma \ref{intense vector space} yields the triviality of $\lambda_q(S_r)$.
\end{proof}

%\noindent
%Thanks to Proposition \ref{no consecutive edges}, the set $\graffe{\pi_s,\pi_t,\pi_0}$ associated to an evolving group is a partition of $V$.

\begin{theorem}\label{structure}
Let $G$ be an evolving group and let $(S_p)_{p\in\pi}$ be a Sylow family of $G$. 
Then one has
\[G=\bigg(\prod_{p\in\pi_t}S_p\rtimes\prod_{p\in\pi_s}S_p\bigg)\times\prod_{p\in\pi_0}S_p.\]
\end{theorem}

\begin{proof}
Let $\Gamma=(V,E)$ be the graph associated to $G$.
We recall that, if $p$ is a prime number and $S$ is a finite $p$-group, then the kernel of the natural map $\Aut(S)\rightarrow\Aut(S/\Phi(S))$ is a $p$-group. It follows that, for a pair $(q,p)\in V^2$ with $q<p$, if $\lambda_p(S_q)$ is trivial, then so is the image of $S_q$ in $\Int(S_p)$. Lemma \ref{no consecutive edges} now yields that, if $\pi_*\in\graffe{\pi_s,\pi_t,\pi_0}$, then any two elements of $(S_q)_{q\in\pi_*}$ centralize each other and that, for $(q,p)\in\pi_0\times\pi$, one has $[S_q,S_p]=\graffe{1}$. 
\end{proof}

\noindent
We close the current section by giving the proof of Theorem \ref{evolvingintense}. The implication $(i)\Rightarrow(ii)$ follows directly from Theorem \ref{structure}, so we prove the other. Let $p$ and $q$ be respectively prime divisors of $|N|$ and $|T|$ and let 
$(S_p,S_q)$ be an element of $\Syl_p(N)\times\Syl_q(T)$. Then, whenever $p<q$, the sequence
$S_q\rightarrow\Int(S_p) \rightarrow\Int(S_p/\Phi(S_p))\rightarrow \F_p^*$
that is derived from Lemmas \ref{intense properties} and \ref{intense vector space} has trivial image. Since the kernel of the map 
$\Int(S_p) \rightarrow\Int(S_p/\Phi(S_p))$ is a $p$-group, it follows that, if the action of $S_q$ on $S_p$ is non-trivial, then $p>q$. We have proven that $G$ is prime-intense and so, by Theorem \ref{behaviour}, evolving.

% ------------------------------------------------------------------------

%\subsection*{Acknowledgment}

\end{document}